\documentclass[12pt,a4paper]{amsart}
\usepackage{graphics}
\usepackage{epsfig}
\usepackage{graphicx}
\usepackage[all]{xy}
\theoremstyle{plain}
\usepackage{amssymb}

\advance\hoffset-20mm \advance\textwidth41mm

\newtheorem{theorem}{Theorem}
\newtheorem{lemma}{Lemma}
\newtheorem*{theo*}{Theorem}
\newtheorem{proposition}{Proposition}

\theoremstyle{definition}
\newtheorem{definition}{Definition}
\newtheorem*{definition*}{Definition}
\newtheorem{example}{Example}
\newtheorem{remark}{Remark}

\def\AA{{\mathbb A}}

\def\QQ{{\mathbb Q}}
\def\KK{{\mathbb K}}
\def\ZZ{{\mathbb Z}}
\def\XX{{\mathbb X}}

\def\NN{{\mathbb N}}

\def\Of{{\mathcal{O}}}

\def\cN{{\mathcal{N}}}
\def\cAT{{\mathcal{A}}}

\def\St{\mathop{\rm Stab}}
\def\div{\mathop{\rm div}}
\def\reg{\mathop{\rm reg}}
\def\sing{\mathop{\rm sing}}
\def\Ker{\mathop{\rm Ker}}
\def\Spec{\mathop{\rm Spec}}

\def\Aut{\mathop{\rm Aut}}
\def\AT{\mathop{\rm AT}}

\def\WDiv{\mathop{\rm WDiv}}
\def\PDiv{\mathop{\rm PDiv}}
\def\id{\mathop{\rm id}}
\def\Cl{\mathop{\rm Cl}}

\begin{document}
\sloppy
\title[Automorphisms of affine toric varieties]
{On orbits of the automorphism group on \\ an affine toric variety}
\author{Ivan Arzhantsev}
\thanks{The first author was partially supported by the Simons Grant and the Dynasty Foundation.}
\address{Department of Higher Algebra, Faculty of Mechanics and Mathematics,
Moscow State University, Leninskie Gory 1, GSP-1, Moscow, 119991,
Russia } \email{arjantse@mccme.ru}
\author{Ivan Bazhov}
\address{Department of Higher Algebra, Faculty of Mechanics and Mathematics,
Moscow State University, Leninskie Gory 1, GSP-1, Moscow, 119991,
Russia } \email{ibazhov@gmail.com}
\date{\today}
\begin{abstract}
Let $X$ be an affine toric variety. The total coordinates on $X$
provide a canonical presentation $\overline{X} \to X$ of $X$
as a quotient of a vector space $\overline{X}$
by a linear action of a quasitorus.
We prove that the orbits of the connected component
of the automorphism group $\Aut(X)$ on $X$ coincide with the Luna strata
defined by the canonical quotient presentation.
\end{abstract}
\maketitle
\section*{Introduction}

Every algebraic variety $X$ carries a canonical stratification by orbits of the automorphism
group $\Aut(X)$. The aim of this paper is to give several characterizations of this stratification
when $X$ is an affine toric variety.

In the case of a complete toric variety $X$ the group $\Aut(X)$ is a linear algebraic group.
It admits an explicit description in terms of combinatorial data defining $X$; see~\cite{De}
and~\cite{Cox}. The orbits of the connected component $\Aut(X)^0$ on the variety $X$
are described in~\cite{Ba}. It is proved there that two points $x,x'\in X$ are in the same
$\Aut(X)^0$-orbit if and only if the semigroups in the divisor class group $\Cl(X)$
generated by classes of prime torus invariant divisors that do not contain $x$ and $x'$
respectively, coincide.

We obtain an analogue of this result for affine toric
varieties. It turns out that in the affine case one may replace the semigroups mentioned above
by the groups generated by the same classes. This relates $\Aut(X)^0$-orbits on $X$ with
stabilizers of points in fibres of the canonical quotient presentation
$\pi: \overline{X} \to \overline{X}/\!/H_X \cong X$, where $H_X$ is a quasitorus
with the character group identified with $\Cl(X)$ and $\overline{X}$ is a
finite-dimensional $H_X$-module whose coordinate ring is the total coordinate ring
(or the Cox ring) of $X$; see~\cite{Cox}, \cite[Chapter~5]{CLS}, or Section~\ref{sec4}
for details. More precisely, our main result (Theorem~\ref{tmain}) states that
the collection of $\Aut(X)^0$-orbits on $X$ coincides with the Luna stratification of $X$
as the quotient space of the linear $H_X$-action on $\overline{X}$.
In particular, in our settings the Luna stratification is intrinsic in the sense of~\cite{KR}.
For connections between quotient presentations of an arbitrary affine variety,
the Luna stratification, and Cox rings, see~\cite{Ar}.

In contrast to the complete case, the automorphism group of an affine toric variety
is usually infinite-dimensional. An explicit description of the automorphism group of an affine toric surface in terms of free amalgamated products is given in~\cite{AZ}.
Starting from dimension three such a description is unknown. Another difference is that
in the affine case the open orbit of $\Aut(X)^0$ coincides with the smooth locus of $X$
\cite[Theorem~2.1]{AKZ}, while for a smooth complete toric variety $X$ the group $\Aut(X)^0$
acts on $X$ transitively if and only if $X$ is a product of projective spaces
\cite[Theorem~2.8]{Ba}.

In Section~\ref{sec1} we recall basic facts on automorphisms of algebraic varieties and define the connected component $\Aut(X)^0$. Section~\ref{sec2} contains a background on affine toric varieties. We consider
one-parameter unipotent subgroups in $\Aut(X)$ normalized by the acting torus (root subgroups).
They are in one-to-one correspondence with the so-called Demazure roots of the cone of the
variety $X$. Also we recall the technique developed in~\cite{AKZ}, which will be used later.
In Section~\ref{sec3} we discuss the Luna stratification on the quotient space of a rational
$G$-module, where $G$ is a reductive group. Necessary facts on Cox rings and canonical quotient
presentations of affine toric varieties are collected in Section~\ref{sec4}. We define the
Luna stratification of an arbitrary affine toric variety and give a characterization
of strata in terms of some groups of classes of divisors (Proposition~\ref{prop23}).
Our main result is proved in Section~\ref{sec5}. We illustrate it by an example. It
shows that although the group $\Aut(X)^0$ acts (infinitely)
transitively on the smooth locus $X^{\reg}$, it may be non-transitive
on the set of smooth points of the singular locus $X^{\sing}$, even when $X^{\sing}$
is irreducible. Finally, in Section~\ref{sec7} we prove collective infinite
transitivity on $X$ along the orbits of the subgroup of $\Aut(X)$ generated by
root subgroups and their replicas (Theorem~\ref{thcit}). Here we follow the approach developed in~\cite{AFKKZ}.

\section{Automorphisms of algebraic varieties}
\label{sec1}

Let $X$ be a normal algebraic variety over an algebraically closed
field $\KK$ of characteristic zero and $\Aut(X)$ be the automorphism group.
At the moment we consider $\Aut(X)$ as an abstract group and our aim
is to define the connected component of $\Aut(X)$ following \cite{Ra}.

\begin{definition}
A family $\{\phi_b\}_{b\in B}$ of automorphisms of a variety $X$, where
the parametrizing set $B$ is an algebraic variety, is an {\it algebraic family}
if the map $B\times X \to X$ given by $(b,x)\to\phi_b(x)$ is a morphism.
\end{definition}

If $G$ is an algebraic group and $G\times X \to X$ is a regular action, then
we may take $B=G$ and consider the algebraic family $\{\phi_g\}_{g\in G}$, where
$\phi_g(x)=gx$. So any automorphism defined by an element of $G$ is included in
an algebraic family.

\begin{definition} \label{def2}
The {\it connected component} $\Aut(X)^0$ of the group $\Aut(X)$ is the subgroup of
automorphisms that may be included in an algebraic family $\{\phi_b\}_{b\in B}$
with an (irreducible) rational curve as a base $B$ such that $\phi_{b_0}=\id_X$ for some
$b_0\in B$.
\end{definition}

\begin{remark}
It is also natural to consider arbitrary irreducible base $B$ in Definition~\ref{def2}.
But for our purposes related to toric varieties rational curves as bases are more suitable.
\end{remark}

It is easy to check that $\Aut(X)^0$ is indeed a subgroup; see~\cite{Ra}.

\begin{lemma}
Let $G$ be a connected linear algebraic group and $G\times X \to X$ be a regular action.
Then the image of $G$ in $\Aut(X)$ is contained in $\Aut(X)^0$.
\end{lemma}

\begin{proof}
We have to prove that every $g\in G$ can be connected with
the unit by a rational curve. By~\cite[Theorem~22.2]{Hu},
an element $g$ is contained in a Borel subgroup of $G$.
As any connected solvable linear algebraic group, a Borel subgroup is isomorphic (as a variety) to
$(\KK^{\times})^r\times\KK^m$ with some non-negative integers $r$ and $m$.
The assertion follows.
\end{proof}

Denote by $\WDiv(X)$ the group of Weil divisors on a variety $X$ and by
$\PDiv(X)$ the subgroup of principal divisors, i.e.,
$$
\PDiv(X)=\{\div(f) \ ; \ f \in \KK(X)^{\times}\} \cup \{0\}.
$$
Then the divisor class group of $X$ is defined as $\Cl(X):=\WDiv(X)/\PDiv(X)$.
The image of a divisor $D$ in $\Cl(X)$ is denoted by $[D]$ and is called the {\it class}
of $D$. Any automorphism $\phi\in\Aut(X)$ acts naturally on the set of prime divisors
and thus on the group $\WDiv(X)$. Under this action a principal divisor goes
to a principal one and we obtain an action of $\Aut(X)$ on $\Cl(X)$.

Recall that the local class group of $X$ in a point $x$ is the factor group
$$
\Cl(X,x) := \WDiv(X) / \PDiv(X, x),
$$
where $\PDiv(X,x)$ is the group of Weil divisors on $X$ that are principle in some
neighbourhood of the point $x$. We have a natural surjection $\Cl(X) \to \Cl(X,x)$.
Let us denote by $\Cl_x(X)$ the kernel of this homomorphism, i.e., $\Cl(X,x)=\Cl(X)/\Cl_x(X)$.
Equivalently, $\Cl_x(X)$ consists of classes that have a representative whose support
does not contain $x$.

\smallskip

We obtain the following result.

\begin{lemma} \label{lemlcg}
Assume that an automorphism $\phi\in\Aut(X)$ acts on $\Cl(X)$ trivially. Then $\Cl_x(X)=\Cl_{\phi(x)}(X)$ for any $x\in X$.
\end{lemma}

\section{Affine toric varieties and Demazure roots}
\label{sec2}

A {\it toric variety} is a normal algebraic variety $X$ containing
an algebraic torus $T$ as a dense open subset such that the action
of $T$ on itself extends to a regular action of $T$ on $X$.
Let $N$ be the lattice of one-parameter subgroups $\lambda: \KK^{\times} \to T$ and
$M=\text{Hom}(N,\ZZ)$ be the dual lattice. We identify $M$
with the lattice of characters $\chi: T\to\KK^{\times}$, and
the pairing $N\times M \to \ZZ$ is given by
$$
(\lambda,\chi) \to
\langle \lambda,\chi\rangle,\quad \text{where} \quad \chi(\lambda(t)):=
t^{\langle \lambda,\chi\rangle}.
$$

Let us recall a correspondence between affine toric varieties and
rational polyhedral cones. Let $\sigma$ be a polyhedral cone in the rational
vector space $N_\QQ:=N\otimes_{\ZZ} \QQ$ and $\sigma^{\vee}$ be the dual cone
in $M_\QQ$. Then the affine variety
$$
X_{\sigma} := \Spec \KK[\sigma^{\vee}\cap M]
$$
is toric and any affine toric variety arises this way, see~\cite{CLS}, \cite{Fu}.
The $T$-orbits on $X_{\sigma}$ are in order-reversing bijection with faces
of the cone $\sigma$.
If $\sigma_0\preceq\sigma$ is a face, then we denote the corresponding $T$-orbit
by $\Of_{\sigma_0}$. In particular, $\Of_{\sigma}$ is
a closed $T$-orbit and, if $o$ is the minimal face of $\sigma$, then
$X_0$ is an open $T$-orbit on $X_{\sigma}$.

An affine variety $X$ is called {\it non-degenerate} if any regular invertible function on
$X$ is constant. If $X$ is toric, this condition means that there are no non-trivial
decompositions $T=T_1\times T_2$ and $X=X_0\times T_2$, where $X_0$ is an affine toric variety
with acting torus $T_1$. If $X=X_{\sigma}$, then $X$ is non-degenerate if and only if
the cone $\sigma$ spans $N_\QQ$ or, equivalently, the cone $\sigma^{\vee}$ is pointed.

Let us consider for a moment the case of an arbitrary toric variety $X$.
Recall that $X$ contains a finite number of prime $T$-invariant divisors
$D_1,\ldots,D_m$. (In the affine case they are in bijection
with rays of the cone $\sigma$.) It is well known that the group $\Cl(X)$ is
generated by classes of these divisors. Let us associate with a $T$-orbit $\Of$
on $X$ the set of all prime $T$-invariant divisors
$D(\Of) := \{D_{i_1},\ldots,D_{i_k}\}$
that do not contain $\Of$. Denote by $G(\Of)$ the subgroup of $\Cl(X)$
generated by the classes of divisors from $D(\Of)$.

\begin{proposition} \label{pr1}
Let $X$ be a toric variety and $x\in X$. Then
$\Cl_x(X)=G(T\cdot x)$.
\end{proposition}

\begin{proof}
Take a class $[D]\in\Cl(X)$ with a $T$-invariant representative $D$.
By definition, $[D]$ is contained in $\Cl_x(X)$ if and only if it contains
a representative $D'\in\WDiv(X)$ whose support does not pass through $x$.
In particular, $D' = D + \div(h)$ for some $h\in\KK(X)$.
Consider the decomposition $D'=D'_+-D'_-$, where $D'_+$ and $D'_-$ are effective.
This allows to deal with only the case where $D'$ is effective.

Suppose that $[D]\in\Cl_x(X)$.
We claim that there exists a $T$-invariant effective divisor $D''\in [D]$
whose support does not pass through $x$. Assume this is not the case and
consider the vector space
$$
\Gamma(X, D') = \{ f\in\KK(X)^{\times} \ ; \ D' + \div(f) \ge 0\} \cup \{0\}.
$$
Then $\Gamma(X,D) = h\Gamma(X,D')$ and the subspace $\Gamma(X,D)$ in $\KK(X)$
is invariant with respect to the action $(t\cdot f)(x):=f(t^{-1}\cdot x)$.
It is well known that $\Gamma(X,D)$ is a rational $T$-module. We can transfer
the structure of rational $T$-module to $\Gamma(X,D')$ by the formula
$$
t\circ f : = h^{-1} t\cdot (hf) \quad \text{for \ any} \quad t\in T, \ f\in\Gamma(X,D').
$$
Then a function $f\in\Gamma(X,D')$ is $T$-semiinvariant if and only if the divisor
$D'+\div(f)$ is $T$-invariant. Since $D'$ is effective, the support of $D'+\div(f)$
passes through $x$ if and only if $f(x)=0$. By our assumption, this is the case for
all $T$-semiinvariants in $\Gamma(X,D')$. But any vector in $\Gamma(X,D')$
is a sum of semiinvariants. Thus the support of any effective divisor in $[D]$
contains $x$. This is a contradiction, because the support of $D'$ does not
pass through $x$.

Since $D''$ is a sum of prime $T$-invariant divisors not passing through $x$,
the class $[D]$, and thus the group $\Cl_x(X)$, is contained in the group
$G(T\cdot x)$. The opposite inclusion is obvious.
\end{proof}

\begin{lemma}
Let $X$ be an affine toric variety.
Then $\Aut(X)^0$ is in the kernel of the action of $\Aut(X)$ on $\Cl(X)$.
\end{lemma}

\begin{proof}
Let $\{\phi_b\}_{b\in B}$ be an algebraic family of automorphisms with
$\phi_{b_0}=\id_X$ for some $b_0\in B$. We may assume that $B$ is an affine
rational curve. In particular, $\Cl(B)=0$. Consider the morphism
$\Phi\colon B\times X \to X$ given by $(b,x)\to\phi_b(x)$.
We have to show that for any divisor $D$ in $X$ the intersections of $\Phi^{-1}(D)$
with fibres $\{b\}\times X$ are linearly equivalent to each other. The torus $T$ acts
on the variety $B\times X$ via the action on the second component. It is well known that
every divisor on $B\times X$ is linearly equivalent to a $T$-invariant one.
Every prime $T$-invariant divisor on $B\times X$ is either vertical, i.e., is a product
of a prime divisor in $B$ and $X$, or horizontal, i.e., a product of $B$ and a prime
$T$-invariant divisor in $X$. Every vertical divisor is principal. So the divisor
$\Phi^{-1}(D)$ plus some principal divisor $\div(f)$ is a sum of horizontal divisors.
Restricting the rational function $f$ to every fibre $\{b\}\times X$ we obtain that
the intersections of $\Phi^{-1}(D)$ with fibres are linearly equivalent to each other.
\end{proof}

Our next aim is to present several facts on automorphisms of affine toric varieties.
Denote by $\sigma(1)$ the set of rays of a cone $\sigma$ and by $v_\tau$ the primitive
lattice vector on a ray $\tau$.

\begin{definition}
An element $e\in M$ is a called a {\it Demazure root} of a polyhedral
cone $\sigma$ in $N_\QQ$ if there is $\tau\in\sigma(1)$ such that
$\langle v_\tau, e\rangle =-1$ and $\langle v_{\tau'}, e\rangle \ge 0$
for all $\tau'\in\sigma(1)\setminus \{\tau\}$.
\end{definition}

Let $\mathcal{R}=\mathcal{R}(\sigma)$ be the set of all Demazure roots of a cone $\sigma$.
For any root $e\in\mathcal{R}$ denote by $\tau_e$ (resp. $v_e$)
the ray $\tau$ (resp. primitive vector $v_\tau$) with $\langle v_\tau, e\rangle =-1$.
Let $\mathcal{R}_\tau$ be the set of roots $e$ with $\tau_e=\tau$.
Then
$$
\mathcal{R} = \cup_{\tau\in\sigma(1)} \mathcal{R}_\tau.
$$
One can easily check that every set $\mathcal{R}_\tau$ is infinite.

With any root $e$ one associates a one-parameter subgroup $H_e$ in
the group $\Aut(X)$ such that $H_e \cong (\KK, +)$ and $H_e$ is normalized
by $T$, see~\cite{De}, \cite{Oda} or \cite[Section~2]{AKZ} for an explicit
form of $H_e$. Moreover, every one-parameter unipotent subgroup of $\Aut(X)$
normalized by $T$ has the form $H_e$ for some root $e$.

The following result is obtained in \cite[Proposition~2.1]{AKZ}.

\begin{proposition} \label{prcon}
Let $e\in\mathcal{R}$. For every point $x\in X\setminus X^{H_e}$ the orbit
$H_e\cdot x$ meets exactly two $T$-orbits $\Of_1$ and $\Of_2$. Moreover, $\Of_2\subseteq\overline{\Of_1}$ and $\dim\Of_1 = 1 + \dim\Of_2$.
\end{proposition}

A pair of $T$-orbits $(\Of_1,\Of_2)$
as in Proposition~\ref{prcon} is called {\it $H_e$-connected}.

\smallskip

The next result is Lemma~2.2 from~\cite{AKZ}.

\begin{lemma} \label{lemcon}
Let $\Of_{\sigma_1}$ and $\Of_{\sigma_2}$ be two $T$-orbits
corresponding to faces $\sigma_1,\sigma_2\preceq\sigma$.
Then the pair $(\Of_{\sigma_1}, \Of_{\sigma_2})$ is $H_e$-connected
if and only if
$$
e|_{\sigma_2} \le 0 \quad \text{and} \ \ \sigma_1=\sigma_2\cap e^{\perp} \ \
\text{is \! a \! facet \! of \! cone} \ \sigma_2.
$$
\end{lemma}

Let $\AT(X)$ be the subgroup of $\Aut(X)$ generated by subgroups $T$
and $H_e$, $e\in\mathcal{R}$.
Clearly, two $T$-orbits $\Of$ and $\Of'$ on $X$ are contained in the same $\AT(X)$-orbit
if and only if there is a sequence $\Of=\Of_1, \Of_2,\ldots,\Of_k=\Of'$ such that
for any $i$ either the pair $(\Of_i,\Of_{i+1})$ or the pair $(\Of_{i+1},\Of_i)$
is $H_e$-connected for some $e\in\mathcal{R}$.

This statement admits a purely combinatorial reformulation. Let $\Gamma(\Of)$ be
a semigroup in $\Cl(X)$ generated by classes of the elements of $D(\Of)$.
The following result is given in Lemmas~2.2-4 of~\cite{Ba}.

\begin{proposition} \label{prat}
Two $T$-orbits $\Of$ and $\Of'$ on $X$ lie in the same $\AT(X)$-orbit
if and only if $\Gamma(\Of)=\Gamma(\Of')$.
\end{proposition}

\section{The Luna stratification}
\label{sec3}

In this section we recall basic facts on the Luna stratification
introduced in~\cite{Lu73}, see also \cite[Section~6]{PV}.
Let $G$ be a reductive affine algebraic group over
an algebraically closed field $\KK$ of characteristic zero and
$V$ be a rational finite-dimensional $G$-module. Denote
by $\KK[V]$ the algebra of polynomial functions on $V$ and
by $\KK[V]^G$ the subalgebra of $G$-invariants. Let $V/\!/G$
be the spectrum of the algebra $\KK[V]^G$. The inclusion
$\KK[V]^G\subseteq\KK[V]$ gives rise to a morphism $\pi \colon
V\to V/\!/G$ called the {\it quotient morphism} for the $G$-module $V$.
It is well known that the morphism $\pi$ is a categorical quotient
for the action of the group $G$ on $V$ in the category of
algebraic varieties, see~\cite[4.6]{PV}. In particular, $\pi$ is surjective.

The affine variety $X:=V/\!/G$ is irreducible and normal.
It is smooth if and only if the point $\pi(0)$ is smooth
on $X$. In the latter case the variety $X$ is an affine
space. Every fibre of the morphism $\pi$ contains a unique closed $G$-orbit.
For any closed $G$-invariant subset $A\subseteq V$ its image $\pi(A)$
is closed in $X$. These and other properties of the quotient morphism
may be found in~\cite[4.6]{PV}.

By Matsushima's criterion, if an orbit $G\cdot v$ is closed
in $V$, then the stabilizer $\St(v)$ is reductive,
see~\cite[4.7]{PV}. Moreover,
there exists a finite collection $\{H_1,\ldots,H_r\}$
of reductive subgroups in $G$ such that if an orbit $G\cdot v$
is closed in $V$, then $\St(v)$ is conjugate to one of these
subgroups. This implies that every fibre of
the morphism $\pi$ contains a point whose stabilizer coincides
with some $H_i$.

For every stabilizer $H$ of a point in a closed $G$-orbit in $V$ the subset
$$
V_H := \{ w\in V \, ; \, \text{there exists}\,
v\in V \, \text{such that} \
\overline{G\cdot w} \supset G\cdot v =\overline{G\cdot v} \
\text{and} \, \St(v)=H\}
$$
is $G$-invariant and locally closed in $V$.
The image $X_H:=\pi(V_H)$ turns out to be a smooth
locally closed subset of $X$. In particular, $X_H$ is
a smooth quasiaffine variety.

\begin{definition} \label{defls}
The stratification
$$
X \, = \, \bigsqcup_{i=1}^r \, X_{H_i}
$$
is called the {\it Luna stratification} of the quotient space $X$.
\end{definition}

Thus two points $x_1,x_2\in X$ are in the same Luna stratum if and only if
the stabilizers of points from closed $G$-orbits in $\pi^{-1}(x_1)$ and
$\pi^{-1}(x_2)$ are conjugate. In particular, if $G$ is a quasitorus,
these stabilizers should coincide.

There is a unique open dense stratum called the {\it principal
stratum} of $X$. The closure of any stratum is a union of strata.
Moreover, a stratum $X_{H_i}$  is contained in the closure of a stratum
$X_{H_j}$ if and only if the subgroup $H_i$ contains a subgroup conjugate to $H_j$.
This induces a partial ordering on the set of strata compatible with the (reverse)
ordering on the set of conjugacy classes of stabilizers.


\section{Cox rings and quotient presentations}
\label{sec4}

Let $X$ be a normal algebraic variety with finitely generated divisor class group $\Cl(X)$.
Assume that any regular invertible function $f\in \KK[X]^{\times}$ is constant.
Roughly speaking, the {\it Cox ring} of $X$ may be defined as
$$
R(X) \, := \, \bigoplus_{[D]\in \Cl(X)} \Gamma(X, D).
$$
In order to obtain a multiplicative structure on $R(X)$ some technical work is needed,
especially when the group $\Cl(X)$ has torsion. We refer to \cite[Section~4]{ADHL} for details.

It is well known that if $X$ is toric and non-degenerate, then $R(X)$ is a polynomial ring
$\KK[Y_1,\ldots,Y_m]$, where the variables $Y_i$ are indexed by $T$-invariant
prime divisors $D_i$ on $X$ and the $\Cl(X)$-grading on $R(X)$ is given by
$\deg(Y_i)=[D_i]$; see~\cite{Cox} and \cite[Chapter~5]{CLS}.

The affine space
$\overline{X}:=\Spec R(X)$ comes with a linear action of a quasitorus $H_X:=\Spec(\KK[\Cl(X)])$
given by the $\Cl(X)$-grading on $R(X)$. The algebra of $H_X$-invariants on $R(X)$
coincides with the zero weight component $R(X)_0=\Gamma(X,0)=\KK[X]$.

Assume that $X$ is a non-degenerate affine toric variety. Then we obtain a
quotient presentation
$$
\pi: \overline{X} \to \overline{X}/\!/H_X \cong X.
$$

\begin{definition}
Let $X$ be a non-degenerate affine toric variety and $V$ be a rational module of
a quasitorus $H$. The quotient morphism $\pi': V\to V/\!/ H$ is called a
{\it canonical quotient presentation} of $X$, if there are an isomorphism
$\varphi: H_X \to H$ and a linear isomorphism $\psi: \overline{X}\to V$
such that $\psi(h\cdot y)=\varphi(h)\cdot\psi(y)$ for any $h\in H_X$ and
$y\in\overline{X}$.
\end{definition}

A canonical quotient presentation may be characterized in terms of the quasitorus action.

\begin{definition}
An action of a reductive group $F$ on an affine variety $Z$ is said to be
{\it strongly stable} if there exists an open dense invariant subset $U\subseteq Z$ such that
\begin{enumerate}
\item
the complement $Z\setminus U$ is of codimension at least two in $Z$;
\item
the group $F$ acts freely on $U$ ;
\item
for every $z\in U$ the orbit $F\cdot z$ is closed in $Z$.
\end{enumerate}
\end{definition}

The following proposition may be found in \cite[Remark~6.4.2 and Theorem~6.4.3]{ADHL}.

\begin{proposition} \label{propstst}
\begin{enumerate}
\item
Let $X$ be a non-degenerate toric variety. Then the action of $H=Spec(\KK[\Cl(X)])$ on
$\overline{X}$ is strongly stable.
\item
Let $H$ be a quasitorus acting linearly on a vector space $V$. Then
the quotient space $X:=V/\!/T$ is a non-degenerate affine toric variety.
If the action of $H$ on $V$ is strongly stable, then the quotient morphism
$\pi: V \to V/\!/T$ is a canonical quotient presentation of~$X$. In particular,
the group $\Cl(X)$ is isomorphic to the character group of $H$.
\end{enumerate}
\end{proposition}

A canonical quotient presentation allows to define a canonical stratification on $X$.

\begin{definition}
The {\it Luna stratification} of a non-degenerate affine toric variety $X$ is the Luna
stratification of Definition~\ref{defls} induced on $X$ by the canonical
quotient presentation $\pi:\overline{X} \to X$.
\end{definition}

\begin{proposition}
Let $X$ be a non-degenerate affine toric variety. Then
the principal stratum of the Luna stratification on $X$
coincides with the smooth locus $X^{\reg}$.
\end{proposition}

\begin{proof}
As was pointed out above, points of the principal stratum are smooth on $X$.
Conversely, the fibre $\pi^{-1}(x)$ over a smooth point $x\in X$ consists
of one $H_X$-orbit and $H_X$ acts on $\pi^{-1}(x)$ freely; see~\cite[Proposition~6.1.6]{ADHL}.
This shows that $x$ is contained in the principal stratum.
\end{proof}

Now we assume that $X$ is a degenerate affine toric variety. Let us fix a point $x_0$
in the open $T$-orbit on $X$ and consider a closed subvariety $X_0=\{x\in X ; f(x)=f(x_0)\}$,
where $f$ runs through all invertible regular functions on $X$. Then $X_0$ is a non-degenerate
affine toric variety with respect to a subtorus $T_1\subseteq T$, and $X_0$ depends on the
choice of $x_0$ only up to shift by an element of $T$. Moreover, $X\cong X_0 \times T_2$
for a subtorus $T_2\subset T$ with $T=T_1\times T_2$. We define a {\it Luna stratum}
on $X$ as $T\cdot Y$, where $Y$ is a Luna stratum on $X_0$.  This way we obtain
a canonical stratification of $X$ with open stratum being the smooth locus.

\smallskip

The following lemma is straightforward.

\begin{lemma} \label{lemred}
In notations as above every Luna stratum on $X$ is isomorphic to
$Y\times T_2$, where $Y$ is a Luna stratum on $X_0$.
\end{lemma}

Now we present the first characterization of the Luna stratification.

\begin{proposition} \label{prop23}
Let $X$ be an affine toric variety. Then two points $x,x'\in X$
are in the same Luna stratum if and only if $\Cl_x(X)=\Cl_{x'}(X)$.
\end{proposition}

\begin{proof}
By Lemma~\ref{lemred} we may assume that $X$ is non-degenerate.
Let $\pi:\overline{X} \to X$ be the canonical quotient presentation.
For any point $v\in\overline{X}$ such that the orbit $H_X\cdot v$ is closed in $\overline{X}$
the stabilizer $\St(v)$ in $H_X$ is defined by the surjection of the character groups
$\XX(H_X) \to \XX(\St(v))$. By~\cite[Proposition~6.2.2]{ADHL} we may identify
$\XX(H_X)$ with $\Cl(X)$, $\XX(\St(v))$ with $\Cl(X,x)$, and the homomorphism
with the projection $\Cl(X)\to\Cl(X,x)$, where $x=\pi(v)$.
Thus two points $v,v'\in \overline{X}$
with closed $H_X$-orbits have the same stabilizers in $H_X$, or, equivalently,
the points $x=\pi(v)$ and $x'=\pi(v')$ lie in the same Luna stratum on $X$ if and only if
$\Cl_x(X)=\Cl_{x'}(X)$.
\end{proof}


\section{Orbits of the automorphism group}
\label{sec5}

The following theorem describes orbits of the group
$\Aut(X)^0$ in terms of local divisor class groups and the Luna stratification
of an affine toric variety $X$.

\begin{theorem} \label{tmain}
Let $X$ be an affine toric variety and $x,x'\in X$. Then
the following conditions are equivalent.
\begin{enumerate}
\item
The $\Aut(X)^0$-orbits of the points $x$ and $x'$ coincide.
\smallskip
\item
$G(T\cdot x) =G(T\cdot x')$.
\smallskip
\item
$\Cl_x(X)=\Cl_{x'}(X)$.
\smallskip
\item
The points $x$ and $x'$ lie in the same Luna stratum on $X$.
\end{enumerate}
\end{theorem}

\begin{proof}
Implication $1 \Rightarrow 3$ follows from Lemma~\ref{lemlcg}.
Conditions $3$ and $4$ are equivalent by Proposition~\ref{prop23}
and conditions $2$ and $3$ are equivalent by Proposition~\ref{pr1}.
So it remains to prove implication $2 \Rightarrow 1$.

\begin{proposition} \label{pr2}
Let $X$ be an affine toric variety and $x\in X$. Then $G(T\cdot x) = \Gamma(T\cdot x)$.
\end{proposition}

\begin{proof}
We begin with some generalities on quasitorus representations.
Let $K$ be a finitely generated abelian group.
Consider a diagonal linear action of the quasitorus $H=\Spec(\KK[K])$ on a vector space
$V$ of dimension $m$ given by characters $\chi_1,\ldots,\chi_m\in K$. Then we have a weight
decomposition $V=\oplus_{i=1}^m \KK e_i$, where $h\cdot e_i=\chi_i(h)e_i$ for any $h\in H$.
With any vector $v=x_1e_1+\ldots+x_me_m$ one associates the set of characters
$\Delta(v)=\{\chi_{i_1},\ldots,\chi_{i_k}\}$ such that $x_{i_1}\ne 0,\ldots,x_{i_k}\ne 0$.
It is well known that the orbit $H\cdot v$ is closed in $V$ if and only if
the cone generated by $\chi_{i_1}\otimes 1,\ldots,\chi_{i_k}\otimes 1$
in $K_\QQ=K\otimes_\ZZ\QQ$ is a subspace.

\smallskip

Below we make use of the following elementary lemma.

\begin{lemma} \label{Lemel}
Let $\chi_1,\dots,\chi_m$ be elements of a finitely generated abelian group $K$.
If the cone generated by $\chi_1\otimes 1,\dots,\chi_m\otimes 1$ in
$K_\QQ$ is a subspace, then the semigroup generated
by $\chi_1,\dots,\chi_m$ in $K$ is a group.
\end{lemma}

Let us return to the canonical quotient presentation $\pi:\overline{X}\to X$.
By construction, the $H_X$-weights of the linear $H_X$-action on $\overline{X}$
are the classes $-[D_1],\ldots,-[D_m]$ in $\Cl(X)$. Moreover,
for any point $v\in\overline{X}$ such that $H_X\cdot v$ is closed in $\overline{X}$
the set of weights $\Delta(v)$ coincides with the classes
of the divisors from $-D(T\cdot x)=\{-D_{i_1},\ldots,-D_{i_k}\}$, where $x=\pi(v)$.
Since the orbit $H_X\cdot v$ is closed in $\overline{X}$, by Lemma~\ref{Lemel}
the semigroup $\Gamma(T\cdot x)$ generated by $[D_{i_1}],\ldots,[D_{i_k}]$ coincides
with the group $G(T\cdot x)$ generated by $[D_{i_1}],\ldots,[D_{i_k}]$, and we
obtain Proposition~\ref{pr2}.
\end{proof}

By Proposition~\ref{pr2}, we have $\Gamma(T\cdot x)=\Gamma(T\cdot x')$.
Further, Proposition~\ref{prat} implies that $x$ and $x'$ lie in the same
$\AT(X)$-orbit and thus in the same $\Aut(X)^0$-orbit. This completes
the proof of Theorem~\ref{tmain}.
\end{proof}

\begin{remark}
Condition 2 of Theorem~\ref{tmain} is the most effective in practice.
It is interesting to know for which wider classes of varieties equivalence
of conditions 1 and 3 holds.
\end{remark}

\begin{remark}
It follows from properties of the Luna stratification that the $\Aut(X)^0$-orbit
of a point $x$ is contained in the closure of the $\Aut(X)^0$-orbit of a point $x'$
if and only if $\Cl_x(X)$ is a subgroup of $\Cl_{x'}(X)$.
\end{remark}

Let us finish this section with a description of $\Aut(X)$-orbits on an affine
toric variety $X$. Denote by $S(X)$ the image of the group $\Aut(X)$ in the
automorphism group $\Aut(\Cl(X))$ of the abelian group $\Cl(X)$. The group
$S(X)$ preserves the semigroup generated by the classes $[D_1],\ldots,[D_m]$
of prime $T$-invariant divisors. Indeed, this is the semigroup of
classes containing an effective divisor.
In particular, the group $S(X)$ preserves the cone in $\Cl(X)\otimes_{\ZZ}\QQ$
generated by $[D_1]\otimes 1,\ldots,[D_m]\otimes 1$. This shows that
$S(X)$ is finite.

\smallskip

The following proposition is a direct corollary of Theorem~\ref{tmain}.

\begin{proposition}
Let $X$ be an affine toric variety. Two points $x,x'\in X$
are in the same $\Aut(X)$-orbit if and only if there exists
$s\in S(X)$ such that $s(\Cl_x(X))=\Cl_{x'}(X)$.
\end{proposition}

If $X$ is an affine toric variety, then the group $\Aut(X)^0$ acts
on the smooth locus $X^{\reg}$ transitively; see~\cite[Theorem~2.1]{AKZ}.
Let $X^{\sing}=X_1\cup\ldots\cup X_r$ be the decomposition of the singular
locus into irreducible components. One may expect that the group $\Aut(X)^0$
acts transitively on every subset $X^{\reg}_i\setminus \cup_{j\ne i} X_j$ too.
The following example shows that this is not the case.

\begin{example}
Consider a two-dimensional torus $T^2$ acting linearly on the vector space $\KK^7$:
$$
(t_1,t_2)\cdot(z_1,z_2,z_3,z_4,z_5,z_6,z_7)=
(t_1z_1,t_1z_2,t_1^{-1}z_3,t_2z_4,t_2z_5,t_1^{-1}t_2^{-1}z_6,t_1^{-1}t_2^{-1}z_7).
$$
One can easily check that this action is strongly stable, and by Proposition~\ref{propstst}
the quotient morphism $\pi: \KK^7 \to \KK^7/\!/T^2$ is the canonical quotient
presentation of a five-dimensional non-degenerate affine toric variety $X:=\KK^7/\!/T^2$.
Looking at closed $T^2$-orbits on $\KK^7$, one obtains that there are three
Luna strata
$$
X = (X \setminus Z) \ \cup \ (Z \setminus \{0\}) \ \cup \ \{0\},
$$
where $Z=\pi(W)$ and $W$ is a subspace in $\KK^7$ given by $z_4=z_5=z_6=z_7=0$.
In particular, $Z$ is the singular locus of $X$. Clearly,
$Z$ is isomorphic to an affine plane with coordinates $z_1z_3$ and $z_2z_3$. So, $Z$
is irreducible and smooth, but the groups $\Aut(X)^0$ (and $\Aut(X)$) has two orbits
on $Z$, namely, $Z \setminus \{0\}$ and $\{0\}$.
\end{example}


\section{Collective infinite transitivity}
\label{sec7}

Let $X$ be a non-degenerate affine toric variety of dimension $\ge 2$.
It is shown in \cite[Theorem~2.1]{AKZ} that for any positive integer $s$
and any two tuples of smooth pairwise distinct points $x_1,\ldots,x_s$
and $x_1',\ldots,x_s'$ on $X$ there is an automorphism $\phi\in\Aut(X)^0$
such that $\phi(x_i)=x_i'$ for $i=1,\ldots,s$. In other words, the
action of $\Aut(X)^0$ on the smooth locus $X^{\reg}$ is {\it infinitely
transitive}.

Our aim is to generalize this result and to prove collective infinite transitivity
along different orbits of some subgroup of the automorphism group.
Let us recall some general notions from~\cite{AFKKZ}.
Consider a one-dimensional algebraic group $H\cong(\KK,+)$ and a regular action
$H\times X \to X$ on an affine variety $X=\Spec A$. Then the associated derivation
$\partial$ of $A$ is locally nilpotent, i.e., for every $a\in A$ we can find
$n\in\NN$ such that $\partial^n(a)=0$. Any derivation of $A$ may be viewed as
a vector field on $X$. So we may speak about locally nilpotent vector fields~$\partial$.
We use notation $H=H(\partial)=\exp(\KK\partial)$.

It is immediate that for every $f\in\Ker\partial$
the derivation $f\partial$ is again locally nilpotent \cite[1.4, Pripciple~7]{Fr}.
A one-parameter subgroup of the form $H(f\partial)$ for some $f\in\Ker\partial$
is called a {\it replica} of $H(\partial)$.

A set $\cN$ of locally nilpotent vector fields on $X$ is said to be {\it saturated}
if it satisfies the following two conditions.
\begin{enumerate}
\item
$\cN$ is closed under conjugation by elements in $G$, where $G$ is the subgroup
of $\Aut(X)$ generated by all subgroups $H(\partial)$, $\partial\in\cN$.
\item
$\cN$ is closed under taking replicas, i.e., for all $\partial\in\cN$ and
$f\in\Ker\partial$  we have $f\partial\in\cN$.
\end{enumerate}

If $X=X_{\sigma}$ is toric, we define $\cAT(X)$ as the subgroup of $\Aut(X)$
generated by $H_e$, $e\in\mathcal{R}$, and all their replicas. As $\cN$ one can take
locally nilpotent vector fields corresponding to replicas of $H_e$ and all their conjugates.

\begin{theorem} \label{thcit}
Let $X$ be a non-degenerate affine toric variety. Suppose that
$x_1,\ldots,x_s$ and $x_1',\ldots,x_s'$ are points on $X$ with
$x_i\ne x_j$ and $x_i'\ne x_j'$ for $i\ne j$ such that for each $i$
the orbits $\cAT(X)\cdot x_i$ and $\cAT(X)\cdot x_i'$ are equal
and of dimension $\ge 2$. Then there exists an element $\phi\in\cAT(X)$
such that $\phi(x_i)=x_i'$ for $i=1,\ldots,s$.
\end{theorem}

The proof of Theorem~\ref{thcit} is based on the following results.
Let $G$ be a subgroup of $\Aut(X)$ generated by subgroups $H(\partial)$, $\partial\in\cN$,
for some set $\cN$ of locally nilpotent vector fields and $\Omega\subseteq X$ be a
$G$-invariant subset.
We say that a locally nilpotent vector field $\partial$ satisfies the
{\it orbit separation property} on $\Omega$ if there is an $H(\partial)$-stable subset
$U(H)\subseteq\Omega$ such that
\begin{enumerate}
\item
for each $G$-orbit $O$ contained in $\Omega$, the intersection $U(H)\cap O$
is open and dense in $O$;
\item
the global $H$-invariants $\KK[X]^H$ separate all one-dimensional $H$-orbits in $U(H)$.
\end{enumerate}

Similarly we say that a set of locally nilpotent vector fields $\cN$
satisfies the {\it orbit separation property} on $\Omega$ if it holds
for every $\partial\in\cN$.

\begin{theorem} (\cite[Theorem~3.1]{AFKKZ}) \label{thAFKKZ}
Let $X$ be an irreducible affine variety and
$G\subseteq\Aut(X)$ be a subgroup
generated by a saturated set $\cN$ of locally nilpotent vector
fields, which has the orbit separation property on a $G$-invariant
subset $\Omega\subseteq X$. Suppose that $x_1,\ldots ,x_s$ and
$x_1',\ldots, x_s'$ are points in $\Omega$ with $x_i\ne x_j$ and
$x_i'\ne x_j'$ for $i\ne j$ such that for each $i$ the orbits
$G\cdot x_i$ and $G\cdot x_i'$ are equal and of dimension $\ge 2$. Then
there exists an element $g\in G$ such that $g\cdot x_i=x'_i$ for
$i=1,\ldots, s$.
\end{theorem}

\begin{proof}[Proof of Theorem~\ref{thcit}.]
Let us take $G=\cAT(X)$ and $\Omega=X$. In order to apply Theorem~\ref{thAFKKZ},
we have to check the
orbit separation property for one-parameter subgroups in $\cAT(X)$.
By~\cite[Lemma~2.8]{AFKKZ}, it suffices to check it for the subgroups
$H_e$, $e\in\mathcal{R}$.

\begin{proposition}
Let $e$ be a root of a cone $\sigma\subseteq N_\QQ$ of full dimension
and $X=X_{\sigma}$ be the corresponding affine toric variety. Then for any
two one-dimensional $H_e$-orbits $C_1$ and $C_2$ there is an invariant
$f\in\KK[X]^{H_e}$ with $f|_{C_1} =0$ and $f|_{C_2} =1$.
\end{proposition}

\begin{proof}
Let $R_e$ be a one-parameter subgroup of $T$ represented by the vector
$v_e\in N$. Then $\KK[X]^{H_e}=\KK[X]^{R_e}$; see~\cite[Section~2.4]{AKZ}.
Moreover, the subgroup $R_e$ normalizes but not centralizes $H_e$ in $\Aut(X)$
and every one-dimensional $H_e$-orbit $C\cong\AA^1$ is the closure of an $R_e$-orbit;
see~\cite[Proposition~2.1]{AKZ}. In particular, every one-dimensional $H_e$-orbit
contains a unique $R_e$-fixed point. Since the group $R_e$ is reductive, every two $R_e$-fixed
points can be separated by an invariant from $\KK[X]^{R_e}$. This shows that any two
one-dimensional $H_e$-orbits can be separated by an invariant from $\KK[X]^{H_e}$.
\end{proof}

This completes the proof of Theorem~\ref{thcit}.
\end{proof}

It follows from the proof of~\cite[Theorem~2.1]{AKZ} that the group $\cAT(X)$
acts (infinitely) transitively on the smooth locus of a non-degenerate affine
toric variety $X$. In particular, the open orbits of $\cAT(X)$ and
$\Aut(X)^0$ on $X$ coincides. The example below shows that this is not the case for
smaller orbits.

\begin{example}
Let $X_{\sigma}$ be the affine toric threefold defined be the cone
$$
\sigma=\rm{cone}(\tau_1,\tau_2,\tau_3), \quad v_{\tau_1}=(1,0,0),
\quad v_{\tau_2}=(1,2,0), \quad v_{\tau_3}=(0,1,2).
$$
We claim that all points on one-dimensional $T$-orbits of $X$ are $\cAT(X)$-fixed.
Indeed, suppose that a point $x$ on a one-dimensional $T$-orbit is moved by
the subgroup $H_e$ for some root $e$. Then $x$ belongs to a union of two
$H_e$-connected $T$-orbits. Assume, for example, that the $T$-orbit of $x$ corresponds
to the face $\rm{cone}(\tau_1,\tau_2)$ and the pair
of $H_e$-connected $T$-orbits includes the $T$-fixed point on $X$. By Lemma~\ref{lemcon},
we have
$$
\langle (1,0,0), e\rangle=0, \quad \langle (1,2,0), e\rangle=0,
\quad \langle (0,1,2), e\rangle=-1.
$$
These conditions imply $\langle (0,0,1), e\rangle=-1/2$, a contradiction.

If the pair of $H_e$-connected $T$-orbits includes the two-dimensional orbit, then
either
$$
\langle (1,0,0), e\rangle=0, \quad \langle (1,2,0), e\rangle=-1, \quad \text{or} \quad
\langle (1,0,0), e\rangle=-1, \quad \langle (1,2,0), e\rangle=0.
$$
In both cases we have $\langle (0,1,0), e\rangle=\pm 1/2$, a contradiction.

Other possibilities may be considered in the same way.
\end{example}


\section*{Acknowledgement}

The first author is grateful to J\"urgen Hausen and Mikhail Zaidenberg
for useful discussions and suggestions.


%
\end{document}